\documentclass[11pt,reqno]{amsart}
\usepackage[utf8]{inputenc}

\usepackage{tikz-cd}
\usepackage{pdfpages}
\usepackage{graphicx}
\usepackage{amssymb,amsmath, amscd, latexsym,amsfonts,bbm, amsthm,stmaryrd, enumerate,phaistos}
\usepackage{yhmath}
\usepackage{comment}\usepackage{todonotes}\usepackage[all]{xy}
\usepackage[vcentermath]{youngtab}
\usepackage{float}

\usepackage{tabularx}
\usepackage{amsmath, amssymb, latexsym, enumerate,  graphicx,tikz, stmaryrd,pifont,ifsym}
\usetikzlibrary{arrows,decorations,decorations.pathmorphing,positioning}
\usepackage{mdframed}
\usepackage{mdwlist}
\usepackage[pdftex]{hyperref}
\usepackage{amscd,mathrsfs,epic,empheq,float}
\usepackage{bbm}
\usepackage{cases}
\usepackage[all]{xy}

\usepackage{tikz}
\usepackage{tikz-cd}

\theoremstyle{definition}

\usepackage{pict2e}
\def\StrangeCross

\usepackage{amsfonts}
\usepackage{amssymb}
\usepackage{amscd}
\usepackage{amsmath}
\def\bea{\begin{eqnarray}}
\def\eea{\end{eqnarray}}

\usepackage{pict2e}

\newtheorem{theorem}{Theorem}[section]

\newtheorem{lemma}[theorem]{Lemma}
\newtheorem{proposition}[theorem]{Proposition}
\newtheorem{problem}[theorem]{Problem}
\newenvironment{exafont}{\begin{bf}}{\end{bf}}

\theoremstyle{definition}
\newtheorem{definition}[theorem]{Definition}
\newtheorem{remark}[theorem]{Remark}
\newtheorem{example}[theorem]{Example}

\newcommand{\Mat}{\operatorname{Mat}}

\begin{document}

\author[V. Gorbounov]{V.~Gorbounov}
\address{V.~G.: Faculty of Mathematics, National Research University Higher School of Economics, Usacheva 6, 119048 Moscow, Russia}
\email{vgorb10@gmail.com }

\author[A. Kazakov]{A.~Kazakov}
\address{A.~K.:  Lomonosov Moscow State University, Faculty of Mechanics and Mathematics, Russia, 119991, Moscow, GSP-1, 1 Leninskiye Gory, Main Building; Centre of Integrable Systems, P. G. Demidov Yaroslavl State University, Sovetskaya 14, 150003, Yaroslavl, Russia;
Center of Pure Mathematics, Moscow Institute of Physics and Technology, 9 Institutskiy per., Dolgoprudny, Moscow Region, 141701, Russian Federation;
Kazan Federal University, N.I. Lobachevsky Institute of Mathematics and Mechanics,  Kazan, 420008, Russia}
\email{anton.kazakov.4@mail.ru}
\title{Electrical networks and data analysis in phylogenetics}

\maketitle

\begin{abstract}

A classic problem in data analysis is studying the systems of subsets defined by either a similarity or a dissimilarity function on $X$ which is either observed directly or derived from a data set.

For an electrical network there are two functions on the set of the nodes defined by the resistance matrix and the response matrix either of which defines the network completely.
We argue that  these functions should be viewed as a similarity and a dissimilarity function on the set of the nodes moreover they are related via the covariance mapping also known as the Farris transform or the Gromov product. We will explore the properties of electrical networks from this point of view.

It has been known for a while that the resistance matrix defines a metric on the nodes of the electrical networks. Moreover for a circular electrical network this metric obeys the Kalmanson property as it was shown recently. We will call such a metric an electrical Kalmanson metric. 
 The main results of this paper is a new description of the electrical Kalmanson metrics in the set of all Kalmanson metrics in terms of the geometry of the positive Isotropic Grassmannian whose connection to the theory of electrical networks was discovered earlier. 

One important area of applications where Kalmanson metrics are actively used is the theory of phylogenetic networks which are a generalization of phylogenetic trees. Our results allow us to use in phylogenetics the powerful methods of reconstruction of the minimal graphs of electrical networks and possibly open the door into data analysis for the methods of the theory of cluster algebras.
\end{abstract}
\setcounter{tocdepth}{3}
\tableofcontents


\section{Introduction}
The theory of electrical networks goes back to the work of Gustav Kirchhoff around mid 1800 and since then it has been a source of remarkable achievements in combinatorics, algebra, geometry, mathematical physics, and electrical engineering. 

An electrical network is a graph with positive weights, conductances attached to the edges, and a chosen subset of the set of vertices which are called the boundary vertices or nodes. An important characteristic of an electrical network with only two nodes $i$ and $j$ is the effective resistance $R_{ij}$, that is the voltage at node $i$ which, when node $j$ is held at zero volts, causes a unit current to flow through the circuit from node $i$ to node $j$. The effective resistance defines a metric on the set of nodes that is widely used in chemistry, for example, \cite{Kl}. 
For convenience, we will organize the effective resistance $R_{ij}$ in a matrix $R$ setting $R_{ii}=0$ for all $i$.

In the case where there are more than two nodes, there is an important generalization of $R_{ij}$. Given an electrical network with $n$ nodes, there is a linear endomorphism of the vector space of functions defined on the nodes, constructed as follows: for each such a function, there is a unique extension of that function to all the vertices which satisfies Kirchhoff's current law at each interior vertex. This function then gives the current $I$ in the network at the boundary vertices defining a linear map which is called the Dirichlet-to-Neumann map or the network response. The matrix of this map is called the response matrix. It plays a key role in the theory and applications of the electrical network \cite{CIM}.

The above two matrices define each other and moreover it is possible to reconstruct a planar circular electrical network if these matrices are known \cite{CIM}.

A classic problem in data analysis is studying the systems of subsets defined by either a similarity or a dissimilarity function on $X$ which is either observed directly or derived from a data set.
While the latter makes use of splits and split metrics, the key ingredients of the former are systems of clusters, subsets of $X$, and elementary similarity functions. One can interpret splits as distinctive features and clusters as common features, see \cite{GandC} and \cite{MS} for an introduction to these ideas.

We argue that  the "inverse" of the response and the resistance matrices should be viewed as a similarity and a dissimilarity function on the set of nodes of an electrical network and as such they are related via the covariance mapping also known as the Farris transform or the Gromov product, see Remark \ref{inverse} for more details. 

We will explore the properties of electrical networks from this point of view.
In this paper we will work with the resistance matrix. The connection of the response matrix to data analysis will be presented in a future publication.

In computational biology species are naturally assigned collections of symbols from a given set, called characters, and one can construct a distance between two species (such as Hamming distance) which records the proportion of characters where the two species differ. Such a record can be encoded by a real symmetric, nonnegative matrix called a dissimilarity matrix $M$. An important problem in phylogenetics is to reconstruct a weighted tree $T$ with the set of leaves equal to the set of of species and the matrix of the tree metric defined by $T$ on the set of leaves is equal to $M$, see \cite{MS} for a nice introduction to these ideas.
 In most cases, a tree structure is too constraining. The notion of a split network is a generalization of a tree in which a certain type of cycles is allowed. A geometric space of such networks was introduced \cite{DP}, forming a natural extension of the work by Billera, Holmes, and Vogtmann on tree space.  
It has been studied from different points of view since it is related to a number of objects in mathematics: the compactification of the real moduli space of curves, to the Balanced Minimal Evolution polytopes and Symmetric Traveling Salesman polytopes to name a few.

The appropriate metric on the set of species defined by a split network has a very special property called the Kalmanson property which distinguishes it completely in the set of all metrics \cite{BD1}, \cite{MS}.

Stefan Forcey has discovered recently that the resistance metric defined by a circular planar electric network obeys the Kalmanson property \cite{F}. We will call such a split metric an {\it electrical Kalmanson metric}. This important result puts a new light on the theory of electrical networks connecting it with phylogenetics and metric geometry.

The purpose of this paper is twofold: firstly we want to collect in one place the relevant facts from the research areas mentioned above and indicate interesting connections between them, secondly we provide a new description of the set of the electrical Kalmanson metrics inside the set of all Kalmanson metrics given in Theorem \ref{th: dual}. For the later we will exploit the connection between the space of circular planar electrical networks and the non-negative Isotropic Grassmannian $\mathrm{IG}_{\geq 0}(n-1, 2n)$ which uses the effective resistance matrix \cite{BGGK}, Theorem $5.6$ as opposed to the response matrix. This description is a modification of the original construction given by Thomas Lam \cite{L} and developed further in \cite{BGKT}, \cite{CGS}.

It turns out that the Kalmanson property itself is a consequence of this connection and the description of the electrical Kalmanson metrics we provide is given entirely in terms of the geometry of the non-negative part of this projective variety. In our description 
checking whether a given Kalmanson metric is an electrical Kalmanson metric amounts to checking that the Plücker coordinates of an explicitly defined point in the above Grassmannian are non-negative. 

Moreover, using the results from \cite{L} we propose an algorithm for reconstructing a minimal circular planar electrical network out of a given resistance matrix which might be useful for possible applications in phylogenetics.

It is remarkable that the theory of positivity might play a role in studying phylogenetic networks since it would allow us to apply the powerful machinery of the theory of cluster algebras developed for describing positivity in mathematical objects \cite{Z}.

This story should be extended to the compactifications of the respective spaces, taking cactus networks, the known strata in the compactification of circular planar electrical networks, to the newly defined compactified split systems as it was already started in \cite{DF}. In this picture the cactus networks should correspond to the pseudometrics playing the role of dissimilarity functions.

The connection of the tropical geometry of the Grassmannians and the space of trees and tree metrics found in \cite{SS} is another interesting direction for developing our work.

Finally building of an example from \cite{F} we show in Remark \ref{nk} that the planarity condition on the network is far from being necessary to guarantee the Kalmanson condition for the resistance metric. It leaves the description of the set of electrical Kalmanson metrics as an interesting question.

{\bf Acknowledgments.} For V. G. this article is an output of a research project implemented as part of the Basic Research Program at the National Research University Higher School of Economics (HSE University). Working on this project V. G. also visited the Max  Plank Institute of Mathematics in the Sciences in Leipzig, Germany in the Fall of 2024 and BIMSA in Beijing, China in the Summer 2024.
Research of A.~K.  on Sections \ref{sec:kalman}  was supported by the Russian Science Foundation project No. 20-71-10110 (https://rscf.ru/en/project/23-71-50012) which finances the work of A.K. at P. G. Demidov Yaroslavl State University.
Research of A.K.  on Section \ref{sec:rec}  was supported by the state assignment of MIPT (project FSMG-2023-0013).

The authors are grateful to Borya Shapiro, Misha Shapiro, Anton Petrunin, Lazar Guterman and especially to Stefan Forcey for useful discussions and helpful suggestions. The authors are grateful to the referees for carefully reading the paper and suggesting a number of useful comments to make it better.

\section{The space of electrical networks and the positive Isotropic Grassmannian}
\subsection{The space of electrical networks}

\begin{definition}
 A  electrical network $\mathcal E$ is a  graph $\Gamma$ with positive $\omega$ weights (conductances) attached to the edges and a chosen subset of the set of vertices which are called the boundary vertices or the nodes. 
\end{definition}
In this paper we will denote by $n$ the number of the nodes.

As explained in the introduction the key result in the theory of electrical networks says the boundary voltages and the currents are related to each other linearly via a  matrix $M_R(\mathcal E)=(x_{ij})$ called the {\it response matrix} of a network. $M_R(\mathcal E)$ has the following properties:
\begin{itemize}
    \item $M_R(\mathcal E)$ is a $n\times n$ symmetric matrix;
    \item All the non-diagonal entries $x_{ij}$ of $M_R(\mathcal E)$ are non-positive;
    \item For each row the sum of all its entries is equal to $0$.
\end{itemize}

Given an electrical network $\mathcal E$ the response matrix $M_R(\mathcal E)$ can be calculated as the Schur complement of a submatrix in the Laplacian matrix of the graph of $\mathcal E$ \cite{CIM}. 
Suppose that $M$ is a square matrix and $D$ is a non-singular square lower right-hand corner submatrix of $M$, so that $M$ has the block structure. 
\[
\begin{pmatrix}
A&B\\
C&D
\end{pmatrix}
\]
The Schur complement of $D$ in $M$ is the matrix $M/D = A - BD^{-1} C$. The Schur complement satisfies the following identity
\[\text{det} M = \text{det} (M/D)\text{det} D\]
Labeling the vertices starting from the nodes we get the Laplacian matrix of the graph representing $\mathcal E$ in a two by two block form as above. The submatrix $D$ corresponds to the connections between the internal vertices and is known to be non degenerate. Then $M_R(\mathcal E)=L/D$.

There are many electrical networks which have the same response matrix, we will describe them now. The following five local network transformations given below are called the {\it electrical transformations}. 
Two electrical networks are said to be equivalent if they can be obtained from each other by a sequence of the electrical transformations. This is an equivalence relation of course, so the set of electrical networks is partitioned in the equivalence classes.
\begin{theorem} \cite{CIM} \label{gen_el_1}
The electrical transformations preserve the response matrix of an electrical network.
\end{theorem}
\begin{figure}[h!]
    \centering
    \includegraphics[width=0.9\textwidth]{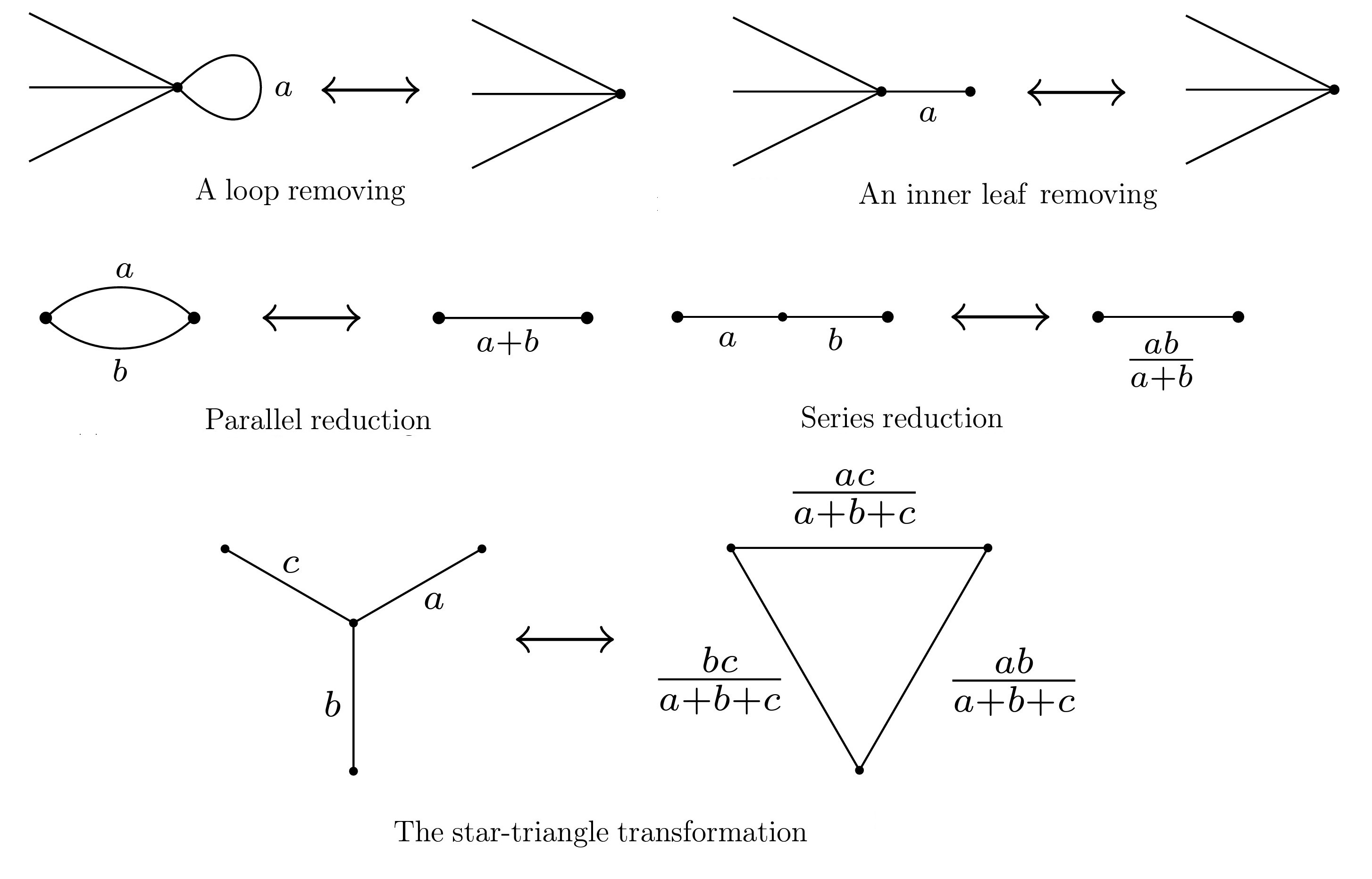}
 \hspace{-1.5cm}   \caption{Electrical transformations }
    \label{fig:el_trans}
\end{figure}

In this paper we will deal with a particular type of connected electrical networks, the {\it connected circular planar electrical networks}, unless otherwise stated. For these we require in addition the graph of $\mathcal E$ to be planar and such that the nodes are located on a circle and enumerated clockwise while the rest of the vertices are situated inside of this circle.

We will denote by $E_n$ the set of equivalence classes of the  circular electrical networks in this paper. Note that the notation for this set of equivalence classes used in \cite{CIM} and in \cite{CIW} is $\Omega_n$. 
The set $E_n$ allows the following elegant description.
\begin{definition}
Let $P=(p_1,\ldots,p_k)$ and $Q=(q_1,\ldots,q_k)$ be disjoint ordered subsets of the nodes arranged on a circle, then $(P;Q)=(p_1,\ldots,p_k;q_1,\ldots,q_k)$ is a {\it circular pair} if $p_1,\ldots, p_k,q_k,\ldots,q_1$ are in non-overlapping circular order around the boundary circle. Let $(P;Q)$ be a circular pair then the determinant of the submatrix $ M(P;Q)$ whose rows are labeled by $(p_1,\ldots,p_k)$ and the columns labeled by $(q_1,\ldots,q_k)$ is called the {\it circular minor} associated with a circular pair $(P;Q)$.
\end{definition}

\begin{theorem} \cite{CIM}, \cite{CdV} \label{Set of response matrices all network}
   The set of response matrices of the the elements of $E_n$ is precisely the set of the matrices $M$ such that
   \begin{itemize}
    \item  $M$ is a symmetric matrix;
    \item All the non-diagonal entries of $M$ are non-positive;
    \item For each row the sum of all its entries is equal to $0$.
    \item For any $k\times k$ circular minor $(-1)^k\det M(P;Q) \geq 0$.
    \item The kernel of $M$ is generated by the vector $(1,1,\dots,1)$.
\end{itemize}
Moreover, any two circular electrical networks which have the same response matrix are equivalent.
\end{theorem}
We will need the following notion later. 
\begin{definition}
Let $\mathcal E$ be a connected circular electrical network on a graph $\Gamma$. The {\it dual electrical network} $\mathcal E^*$ is defined in the following way
\begin{itemize}
\item the graph $\Gamma^*$ of $\mathcal E^*$ is the dual graph to $\Gamma$; 
\item for any pair of dual edges of $\Gamma$ and $\Gamma^*$ their conductances are reciprocal to each other;
\item the labeling of the nodes of $\mathcal E^{*}$ is determined by the requirement that the first node of $\mathcal E^{*}$ lies between the first and second node of $\mathcal E$.
\end{itemize}
\end{definition}
\subsection{Cactus electrical networks}

Setting an edge conductance to zero or infinity in $\mathcal E$ makes sense. According to the Ohm law it means that we delete or contract this edge. Doing it we either get isolated nodes or some nodes get glued together. We will consider the resulting network as a network with $n$ nodes, remembering how nodes get identified. Such a network is called {\it a cactus electrical network} with $n$ nodes. One can think of it as a collection of ordinary circular electrical networks with the total number of the nodes equal to $n$, glued along some of these nodes. Note, that the graph of such a network is planar, but it does not have to be connected.
  
\begin{figure}[h!]
    \centering
    \includegraphics[width=1.0\textwidth]{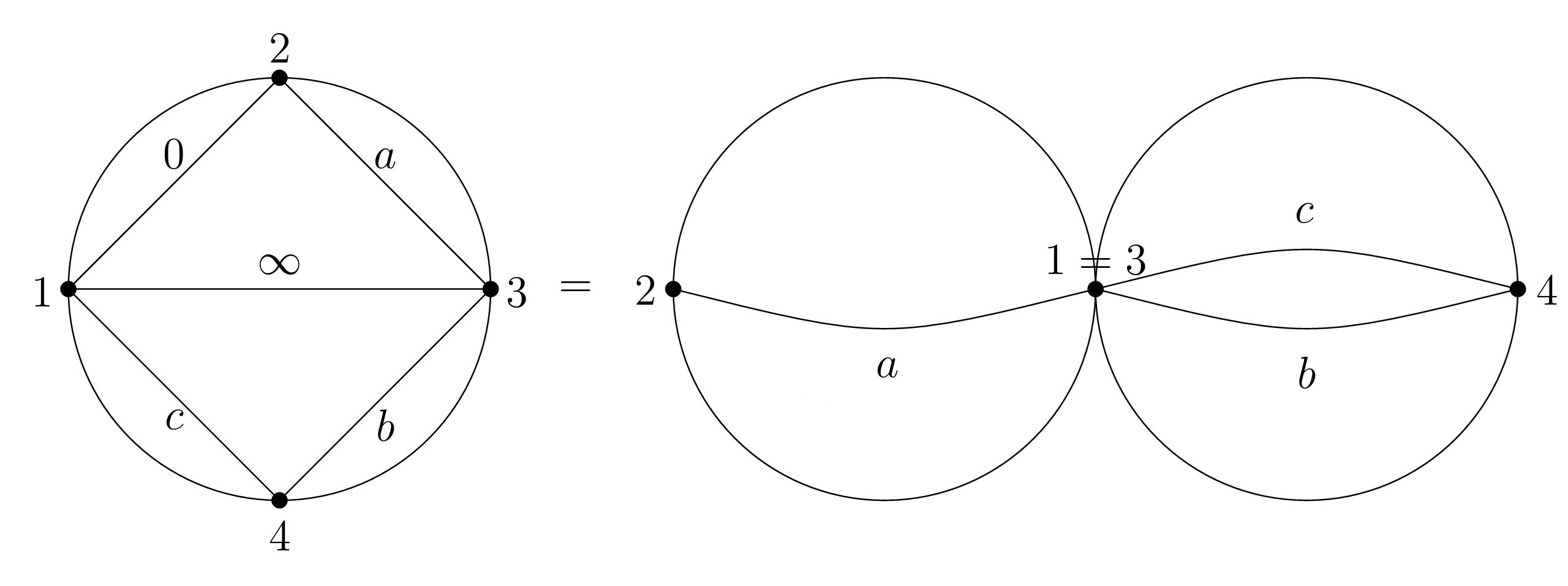}
    \caption{A cactus electrical network with 4 nodes }
    \label{fig:cactus}
\end{figure}

The electrical transformations can be applied to the cactus electrical networks. We denote by $\overline{E}_n$ the set of equivalence classes with respect to the electrical transformations of the cactus electrical networks with $n$ nodes.

The definition of a cactus network was introduced in \cite{L} where it was proved that the set  $\overline{E}_n$ is a compactification of $E_n$ in the appropriate sense. Note that the space of cactus networks in denoted by $E_n$ in \cite{L}. We apologize for this unfortunate clash of notations.

\subsection{Lam embedding} \label{sec: lam emb} 

Recall that the real Grassmannian $\mathrm{Gr}(k, n)$
is a differentiable manifold that parameterizes the set of all 
$k$-dimensional linear subspaces of a vector space $\mathbb{R}^n$. In fact it has a structure of a projective algebraic variety. The Plücker embedding  is an embedding of the Grassmannian $\mathrm{Gr}(k, n)$ into the projectivization of the $k$-th exterior power of the vector space $\mathbb{R}^n$
\[\iota :\mathrm {Gr} (k,n)\rightarrow \mathrm {P}(\Lambda ^{k}\mathbb{R}^n)\]
Suppose $W\subset \mathbb{R}^n$ is a $k$-dimensional subspace. To define 
$\iota (W)$, choose a basis
$(w_{1},\cdots ,w_{k})$ for 
$W$, and let 
$\iota (W)$ be the projectivization of the wedge product of these basis elements:
$\iota (W)=[w_{1}\wedge \cdots \wedge w_{k}]$, where 
$ [\,\cdot \,]$ denotes the projective equivalence class. 

For practical calculations one can view the matrix whose rows are the coordinates of the basis vectors $(w_{1},\cdots ,w_{k})$ as a representative of this equivalence class.
For any ordered sequence $I$ of $k$ positive integers
$1\leq i_{1}<\cdots <i_{k}\leq n$ denote by $\Delta_I$ the determinant of a $k\times k$ submatrix of the above matrix with columns labeled by the numbers from $I$. The numbers $\Delta_I$ are called the Plücker coordinates of the point $W$ of $\mathrm Gr(k,n)$. They are defined up to a common non zero factor.

\begin{definition}
The totally non-negative Grassmannian $\mathrm{Gr}_{\geq 0}(k, m)$ is the subset of the points of the Grassmannian $\mathrm{Gr}(k, n)$ whose Plücker coordinates $\Delta_I$ have the same sign or equal to zero.
\end{definition}

The following theorem of T. Lam  \cite{L} is one of the key results about the space of electrical networks.
\begin{theorem} \label{th: main_gr}
There is a bijection 
\[\overline {E}_n \cong \mathrm{Gr}_{\geq 0}(n-1,2n)\cap \mathbb{P}H\]
where $H$ is a certain subspace of $\bigwedge^{n-1}\mathbb{R}^{2n}$ of dimension equal to the Catalan number $C_n$.  

Moreover, the image of the set $E_n$ under this bijection is exactly the set of points with the Plücker coordinates $\Delta_{24\dots 2n-2}$ and $\Delta_{13\dots 2n-3}$ not equal to zero.
\end{theorem}

We will recall the explicit construction of the embedding of ${E}_n$ obtained in \cite{BGKT}, which is induced by the above bijection.
Let $\mathcal E$ be a circular electrical network with the response matrix $M_R(\mathcal E)=(x_{ij}).$  The following $n\times 2n$ matrix
\begin{equation} \label{omega_eq}
\Omega(\mathcal E)=\left(
\begin{array}{cccccccc}
x_{11} & 1 & -x_{12} & 0 & x_{13} & 0 & \cdots & (-1)^n\\
-x_{21} & 1 & x_{22} & 1 & -x_{23} & 0 & \cdots & 0 \\
x_{31} & 0 & -x_{32} & 1 & x_{33} & 1 & \cdots & 0 \\
\vdots & \vdots &  \vdots &   \vdots &  \vdots & \vdots & \ddots & \vdots 
\end{array}
\right)    
\end{equation}
gives the point in $\mathrm{Gr}_{\geq 0}(n-1,2n)\cap \mathbb{P}H$ which corresponds to $\mathcal E$ under the Lam bijection. 

Since the sum of the entries in each row of $M_R(\mathcal E)$ is equal to $0$ (see Theorem \ref{Set of response matrices all network}), we have that  $\Omega(\mathcal E)$ has the rank equal to $n-1$. Therefore, the dimension of the row space of $\Omega(\mathcal E)$ is equal to $n-1$, hence it defines a point in $\mathrm{Gr}(n-1,2n)$. The Plücker coordinates of the point associated with $\Omega(\mathcal E)$ can be calculated as the maximal size minors of the matrix $\Omega'(\mathcal E)$ obtained from $\Omega(\mathcal E)$  by deleting, for example, the first row.
\begin{theorem} \cite{BGK} \label{about sur}
Let $A=(a_{ij})$ be a matrix which satisfies the first three conditions of Theorem \ref{Set of response matrices all network} and $\Omega(A)$ be a matrix constructed  according to the  formula \eqref{omega_eq}. 
If  $\Omega(A)$ defines a  point  in  $Gr_{\geq 0}(n-1, 2n)$  and the Plücker coordinate $\Delta_{13\dots 2n-3}\bigl(\Omega(A)\bigr)$ is not equal to zero, then there is a connected electrical network $\mathcal E \in E_n$ such that 
$A=M_R(\mathcal E).$
\end{theorem}

\begin{example}
 For the network  $\mathcal E \in E_4$  seen in the Figure \ref{treedaul} the matrix $\Omega(\mathcal E)$ has the following form:
 \begin{equation*}
\Omega(\mathcal E) = \left(
\begin{array}{cccccccc}
\dfrac{5}{8}& 1  & \dfrac{1}{8}  & 0 &  -\dfrac{1}{8} & 0 & \dfrac{3}{8} & 1  \\
&  &   &  &   &  &  &  \\
\dfrac{1}{8}& 1  & \dfrac{5}{8}  & 1 &  \dfrac{3}{8} & 0 & -\dfrac{1}{8} & 0  \\
&  &   &  &   &  &  &  \\
-\dfrac{1}{8}& 0  &\dfrac{3}{8}  & 1 & \dfrac{5}{8} & 1  & \dfrac{1}{8} & 0 \\
&  &   &  &   &  &  &  \\
\dfrac{3}{8}& 0  & -\dfrac{1}{8}  & 0 &  \dfrac{1}{8} & 1  & \dfrac{5}{8} & 1 \\
\end{array}
\right).
\end{equation*}
\end{example}

In fact, the row space of $\Omega(\mathcal E)$ is isotropic with respect to a particular symplectic form. This refines the above embedding to a submanifold 
$$\mathrm{IG}_{\geq 0}(n-1, 2n)\subset \mathrm{Gr}_{\geq 0}(n-1, 2n)$$
made out of isotropic subspaces of $\mathbb{R}^{2n}$ \cite{BGKT}, \cite{CGS}.

\section{Characterization of resistance distance} \label{sec:kalman}
\subsection{Resistance metric}
Now we assume that an electrical network is just connected.
\begin{definition}  \label{def:eff-resist}
    Let $\Gamma, \omega$ be a connected  graph  with a weight function $\omega$ on the edges and $i,\, j\in \Gamma$ be two of its vertices. Consider it as an electrical network  $\mathcal E_{ij}(\Gamma, \omega)$ on a graph $\Gamma$ with a conductivity function $\omega$  by declaring the vertices  $i$ and  $j$ to be the nodes while the remaining vertices to be internal vertices.
    
Apply the voltages $U=(U_i, U_j)$ to the nodes of $\mathcal E_{ij}(\Gamma, \omega)$ such that they induce the boundary currents $I=(1, -1)$.  Then the {\it effective resistance} between the nodes  $i$ and  $j$ is defined  as 
    $$R_{ij}=|U_i-U_j|$$
Obviously $R_{ij}=R_{ji}$.  
\end{definition}

This definition can be stated which is due to Gustav Kirchhoff in a beautiful way using the combinatorics of the graph. Two pieces of notation are required: 
recall the generating function of spanning-trees of $\Gamma$ is defined as follows
\[T(\Gamma,\omega) = \sum_{T\in T(\Gamma)} \omega(T),\]
where the sum is taken over the set $T(\Gamma)$ of all spanning trees of $\Gamma$ and $\omega(T)$ is defined as the product of the weights of the edges of $T$. 
For a graph $\Gamma, \omega$ as above and vertices $i,j \in V$, we let $\Gamma/ij$ denote the graph obtained by merging the two vertices $i$ and $j$ together into a single vertex. 
\begin{theorem}\label{KF} (Kirchhoff’s Formula \cite{MTT}, \cite{Wa}). Let $(\Gamma,\omega)$ be an electrical network, and let $i,j \in V$. The effective resistance between $i$ and $j$ in $\Gamma$ is 

\[R_{ij} =\frac{T(\Gamma/ij;\omega)} {T(\Gamma;\omega)} \]
\end{theorem}
The following lemma relates the effective resistances and the response matrix entries is well known. 
\begin{lemma} \cite{KW 2011} \label{lem:eff-resist}
Let $\mathcal E$ be a connected  electrical network with $n$ nodes\\ $\{1, \dots, i, \dots, j, \dots, n\} $, and let the boundary voltages  $U = (U_1, \dots , U_n)$ be such that
\begin{equation} \label{eq-resist}
    M_R(\mathcal E)U = -e_i + e_j,
\end{equation}
 where $e_k, \ k \in \{1, \dots  , n\}$ is the  standard basis of $\mathbb{R}^n$.
Then $$|U_i - U_j|=R_{ij}.$$
\end{lemma}
\begin{proof}
    Indeed, if the boundary voltages $U$ are such as in \eqref{eq-resist} we can consider all the vertices except $i$ and $j$ as the inner vertices, and then $|U_i - U_j|$ is precisely as in Definition \ref{def:eff-resist}.  
\end{proof}
For convenience, we will organize the effective resistances $R_{ij}$ between the nodes of $\mathcal E$ in a symmetric matrix $R_{\mathcal E}$ setting $R_{ii}=0$ for all $i$.
We call this matrix the {\it resistance matrix} of $\mathcal E$ and denote it by $R_{\mathcal E}$. 

From Lemma \ref{lem:eff-resist} it follows that 
$$R_{ij}=(-e_i+e_j)^t\bigl(M_R(\mathcal E)\bigr)^{-1}(-e_i+e_j),$$
where $\bigl(M_R(\mathcal E)\bigr)^{-1}(-e_i+e_j)$ means a vector $U$ which satisfies \eqref{eq-resist}. Notice that such a vector always exists. 

\begin{proposition}\cite{KW 2011} \label{th: about inverse resp} 
Let $\mathcal E$ be a connected electrical network. Denote by $M'_R(\mathcal E)$ the matrix obtained from $M_R(\mathcal E)$ by deleting  the last row and the last column, then $M'_R(\mathcal E)$ is invertible. The matrix elements of its inverse are given by the formula
 \begin{equation*}
M'_R(\mathcal E)^{-1}_{ij}=\begin{cases}
   R_{in},\, \text{if}\,\, i=j   \\
   \frac{1}{2}(R_{in}+R_{jn}-R_{ij}),\, \text{if}\,\, i\not = j,\\
    
      \end{cases}
    \end{equation*} 

\end{proposition}

Proposition \ref{th: about inverse resp} and Lemma \ref{eq-resist} show that the resistance and the response matrices define each other. Therefore, Theorem \ref{gen_el_1} would hold if we replace the response matrix with the resistance matrix in its statement.

\begin{remark}\label{inverse} The formula from Proposition \ref{th: about inverse resp} is well known in different areas of mathematics. It appeared in the literature under the names the Gromov product, the Farris transform, the Covariance mapping between the Cut and the Covariance cones, see \cite{GandC} for more information. This allows us to view $M'_R(\mathcal E)^{-1}$, the "inverse" of the response matrix, as a similarity function corresponding under the Covariance mapping to a dissimilarity function represented by the resistance metric $R_{\mathcal E}$. We are planning to explore the properties of the resistance matrix from these points of view in the future publications.
\end{remark}
Proposition \ref{th: about inverse resp} provides a simple proof that the effective resistances $R_{ij}$ satisfy the triangle inequality.
\begin{theorem} \label{th:about metric}
    Let $\mathcal E$ be an  electrical network on a connected network $\Gamma$ then for any of three nodes $k_1, k_2$ and $k_3$ the triangle inequality holds:
    $$R_{k_1k_3}+R_{k_2k_3}-R_{k_1k_2} \geq 0.$$
    Hence the set of all $R_{vw}$ defines a metric on the nodes of $\Gamma$.
   
\end{theorem}
\begin{proof} 
     Let $\mathcal E_{k_1k_2k_3}$ be a connected electrical network obtained from $\mathcal E$ by declaring the vertices $k_1, k_2, k_3$ to be the boundary nodes, while remaining vertices are declared to be inner and $M_R(\mathcal E_{k_1k_2k_3})$ be its response matrix, then according to Proposition \ref{th: about inverse resp} we have that:
    $$M'_R(\mathcal E_{k_1k_2k_3})^{-1}_{k_1k_2}=\frac{1}{2}(R_{k_1k_3}+R_{k_2k_3}-R_{k_1k_2}),$$
    therefore  to get the statement it is enough to verify that  $M'_R(\mathcal E_{k_1k_2k_3})^{-1}_{k_1k_2} \geq 0.$ 
Indeed, the matrix $M'_R(\mathcal E_{k_1k_2k_3})$ has the following form:
 \begin{equation*} 
M'_R(\mathcal E_{k_1k_2k_3})=\left(\begin{matrix}
x_{k_1k_1} & x_{k_1k_2}   \\
 x_{k_1k_2} &  x_{k_2k_2}      
\end{matrix} \right) = \left(\begin{matrix}
-x_{k_1k_2}-x_{k_1k_3} & x_{k_1k_2}   \\
 x_{k_1k_2} &  -x_{k_1k_2}-x_{k_2k_3}      
\end{matrix} \right).
\end{equation*}
    By the direct computation we obtain that 
    \begin{equation*}
        M'_R(\mathcal E_{k_1k_2k_3})^{-1}_{k_1k_2}=\frac{-x_{k_1k_2}}{\det  M'_R(\mathcal E_{k_1k_2k_3}) }=\frac{-x_{k_1k_2}}{x_{k_1k_2}x_{k_2k_3}+x_{k_1k_3}x_{k_1k_2}+x_{k_1k_3}x_{k_2k_3}},
    \end{equation*}
    By the definition of the response matrix all $x_{k_ik_j} \leq 0, i\neq j$ which implies the statement of the theorem. 
\end{proof}

To describe the properties of the resistance metric associated with a connected circular electrical network, we will provide a formula for the embedding of $E_n$ into $\mathrm{Gr}_{\geq 0}(n-1, 2n)$ described in Theorem \ref{th: main_gr} which uses the effective resistances matrix instead of the response matrix.

Let $\mathcal E$ be a connected network with the resistance matrix $R_{\mathcal E}$, define a point in $Gr(n-1,2n)$ associated to it as the row space of the matrix:
\begin{equation} \label{eq:omega_n,r}
  \Omega_{R}(\mathcal E)=\left(\begin{matrix}
1 & m_{11} & 1 &  -m_{12} & 0 & m_{13} & 0 & \ldots  \\
0 & -m_{21} & 1 & m_{22} & 1 & -m_{23} & 0 & \ldots \\
0 & m_{31} & 0 & -m_{32} & 1 & m_{33} & 1 & \ldots \\
\vdots & \vdots & \vdots & \vdots & \vdots & \vdots &  \vdots & \ddots 
\end{matrix}\right),
\end{equation}
where 
$$m_{ij}= -\frac{1}{2}(R_{i,j}+R_{i+1,j+1}-R_{i,j+1}-R_{i+1,j}).$$
Notice that the matrix $M(R_{\mathcal E})=(m_{ij})$ is symmetric and the sum of the matrix entries in each row is zero. In other words, it looks like a response matrix of an electrical network.
There is a reason for this as it was discovered by R. Kenyon and D. Wilson.
\begin{theorem} \cite{KW 2011} \label{ken-wen}
    Let $\mathcal E$ be a connected circular electrical network and $\mathcal E^{*}$ be its dual, then the following holds:
    \begin{equation} \label{form:xij}
    x^{*}_{ij}=-\frac{1}{2}(R_{i,j}+R_{i+1,j+1}-R_{i,j+1}-R_{i+1,j})
    \end{equation}
     \begin{equation}  \label{form:rij}
   R^{*}_{ij}=-\sum \limits_{i'<j': \ D_{S_{ij}}(i', j') \neq 0}x_{i'j'},
    \end{equation}
where $x^*_{ij}$ are the matrix elements of the response matrix of the dual network $\mathcal E^*$. 
\end{theorem}

Since $M(R_{\mathcal E})$ is a degenerate matrix, the Plücker coordinates of a point $Gr(n-1, 2n)$ associated with $\Omega_{R}(\mathcal E)$ can be calculated as the maximal size minors of the matrix $\Omega'_{ R}(\mathcal E)$ obtained from $\Omega_{R}(\mathcal E)$  by deleting, for example, the last row.

\begin{theorem}  \cite{BGGK} \label{th:aboout omeganr}
The row space of $\Omega_{R}(\mathcal E)$ defines the same point in the $Gr_{\geq 0}(n-1,2n)$ as the point $\Omega(\mathcal E)$ defined by the Lam embedding, see Theorem \ref{th: main_gr}. In particular, the Plücker coordinate $\Delta_{24\dots 2n-2}(\Omega_{R}(\mathcal E))$ is not 0 if and only if $\mathcal E$ is a connected circular electrical network.

Putting it together
\[ 
\Omega_{R}(\mathcal E)=\Omega(\mathcal E^{*})s=\Omega(\mathcal E)\] 
where $s \in \Mat_{2n \times 2n}$ is the shift operator
\[s=\left(\begin{matrix}
0 & 1 &  0 & 0 & \cdots & 0 \\
0 & 0 & 1 & 0 & \cdots & 0 \\
0 & 0 & 0 & 1 & \cdots & 0 \\
\vdots & \vdots & \vdots & \vdots & \ddots & \vdots \\
0 & 0 & 0 & 0 & \cdots & 1 \\
(-1)^{n} & 0 & 0 & 0 & \cdots & 0
\end{matrix}\right),\]\label{cyclic operator s}
See the proof of \cite[Theorem 5.6]{BGGK} for more details.
\end{theorem}

Many interesting inequalities involving $R_{ij}$ follow from the positivity of the Plücker coordinates of the point represented by $\Omega_{R}(\mathcal E)$. For some of them we have found an explicit combinatorial meaning, others are still waiting to be interpreted.

Below we will deduce the Kalmanson property for the metric $R_{\mathcal E}$ as a consequence of the positivity described above. This fact was established in  \cite{F} earlier using different methods.   

\begin{theorem} \label{charkalm}
 Let $\mathcal E$ be a connected circular electrical network and let  $i_1, i_2, i_3, i_4$ be any four nodes in the circular order as it is shown in the Figure \ref{kalmanson}. Then the  Kalmanson inequalities hold:
      \begin{equation} \label{kal_1}
      R_{i_1i_3}+R_{i_2i_4}\geq R_{i_2i_3}+R_{i_1i_4},     
      \end{equation}
      \begin{equation} \label{kal_2}
      R_{i_1i_3}+R_{i_2i_4}\geq R_{i_1i_2}+R_{i_3i_4}.    
      \end{equation}
\end{theorem}

\begin{proof}
    Let $\mathcal E_{i_1i_2i_3i_4}$ be an electrical network obtained from $\mathcal E$ by declaring the vertices $i_1, i_2, i_3, i_4$ the boundary nodes, and the rest of the vertices are internal. Obviously $\mathcal E_{i_1i_2i_3i_4}$ is a connected  circular electrical network in $E_4$ and let  $\Omega'_{R}(\mathcal E_{i_1i_2i_3i_4})$ be its matrix defined by the formula \ref{eq:omega_n,r}.
    By direct computations we conclude that

    $$\Delta_{567}\bigl(\Omega'_{ R}(\mathcal E_{i_1i_2i_3i_4})\bigr)=\frac{1}{2}(R_{i_1i_3}+R_{i_2i_4} -R_{i_1i_2}-R_{i_3i_4}),$$
     $$\Delta_{123}\bigl(\Omega'_{R}(\mathcal E_{i_1i_2i_3i_4})\bigr)=\frac{1}{2}(R_{i_1i_3}+R_{i_2i_4}- R_{i_2i_3}-R_{i_1i_4}).$$
     Taking into account that all the minors $\Delta_{I}\bigl(\Omega'_{ R}(\mathcal E_{i_1i_2i_3i_4})\bigr)$ are non-negative we obtain that the Kalmanson inequalities hold for the metric defined by the resistances.
\end{proof}
    \begin{figure}[h!]
\center
\includegraphics[width=60mm]{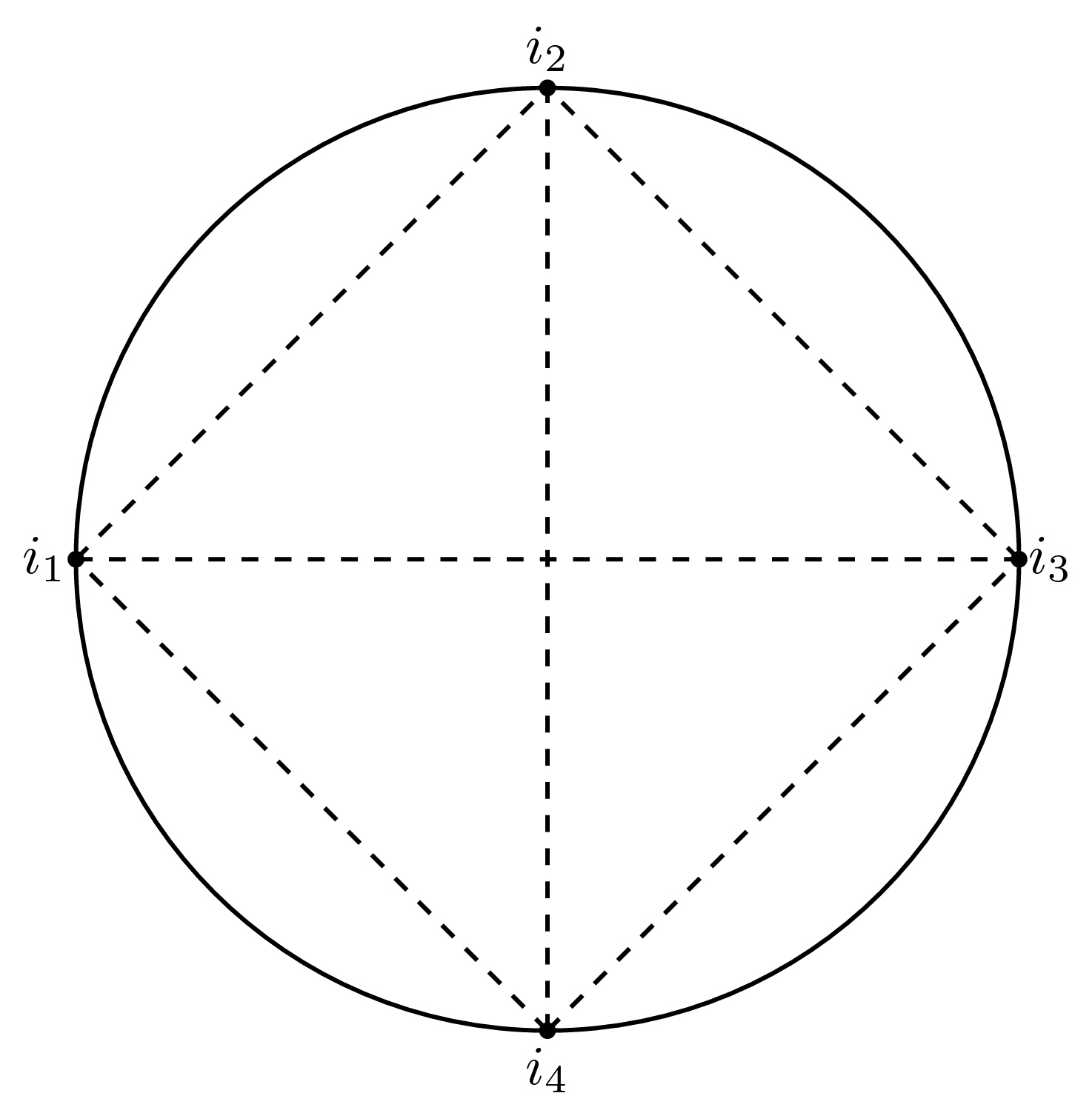}
\caption{$4$-point Kalmanson property}
\label{kalmanson}
\end{figure}
\begin{remark}
Note that the equations \ref{kal_1}, \ref{kal_2} together are equivalent to the tropical inequality:
\[
    R_{i_{1}i_{3}}+R_{i_{2}i_{4}}\geq\max(R_{i_{2}i_{3}}+R_{i_{1}i_{4}},\,R_{i_{1}i_{2}}+R_{i_{3}i_{4}})
    \]
\end{remark}
\begin{definition} Let $X$ be a finite set and $D$ is a metric on it. If there is a circular order on $X$ such that the inequalities from the theorem above hold for any four points of $X$ in this circular order, we call this metric the {\it Kalmanson metric}.
\end{definition}
Therefore the resistance metric $R_{\mathcal E}$ defined by a circular electrical network $\mathcal E$ is a Kalmanson metric.

\subsection{Circular split systems and electrical networks} \label{sec:splitmetr}
\begin{definition}
    A {\it split} $S$ of a set $X=\{1, \dots, n\}$  is a partition of  $X$ into two non-empty, disjoint subsets  $A$ and $B$, $ A \sqcup B = X$. A split is called trivial if either $A$ or $B$ has cardinality $1.$
    A  collection of splits $\mathcal{S}$ is called a {\it split system }.
\end{definition}

The pseudo metric associated with a split $S$  is defined by the following matrix $D:$
    \begin{equation*}
        D_{A|B}(k,l)=\begin{cases}
             1, \ \text{if} \ |A \cap \{k,l\}|=1, \\
             0, \ \text{otherwise.}
        \end{cases}
    \end{equation*}

A circular order of $X$ can be drawn as a polygon with the elements of $X$ labeling the sides. A {\it circular split system} is a split system for which a circular order exists such that all the splits can be simultaneously drawn as sides or diagonals of the labeled polygon. Trivial splits are sides of the polygon, separating the label of that side from the rest, while a non-trivial split $A|B$ is a diagonal separating the sides labeled by $A$ and $B$. For any circular split system we can visualize it by such a polygonal representation, or instead choose a visual representation using sets of parallel edges for each split; these representations are called {\it circular split networks}. A set of parallel edges displays a split $A|B$ if the removal of those edges leaves two connected components with respective sets of terminals $A$ and $B$, see the Figure \ref{splex}. 
\begin{definition} We call a {\it weighted circular split system} a circular split system with a positive weight attached to each set of parallel edges, which display the splits. Such a weighted circular split system defines a pseudo metric on the set $X$ by the formula
\begin{equation} \label{decom1}
        D_S=\sum \limits_{A|B \in S } \alpha_{A|B}D_{A|B},
    \end{equation}
where $\alpha_{A|B}$ is the weight of the set of parallel edges which defines the split $A|B$.
\end{definition}
The split $(\{1,2,3,4\}|\{5,6,7,8\})$ of the set of the sides of the polygon on the left  defined by the yellow diagonal in the Figure \ref{splex} corresponds to the set of parallel edges with the weight $\alpha_1$ in the graph on the left.
\begin{figure}[h!]
\center
\includegraphics[width=90mm]{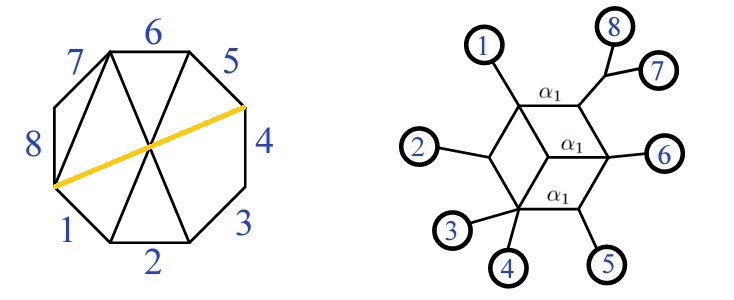}
\caption{Circular split system and their polygon representations. Removing any set of parallel edges defines a split of the set of leaves of the graph on the right which coincides with the split of the set of the sides of a polygon on the left obtained by removing an appropriate diagonal.}
\label{splex}
\end{figure}

The following theorem gives a beautiful characterization of the set of Kalmanson metrics \cite{BD}.
\begin{theorem} \label{BD} A metric $d$ is a Kalmanson metric with respect to a circular order $c$ if and only if $d = D_S$ for a unique weighted circular split system $S$, (not necessarily containing all trivial splits) with each split $A|B$ of $S$ having both parts contiguous in that circular order $c$.
\end{theorem}

Since the resistance metric defined by a planar circular electrical network satisfies the Kalmanson condition, it corresponds to a unique weighted circular split system, which we call an {\it electrical circular split system} following \cite{F} or, given Theorem \ref{charkalm}, an {\it electrical Kalmanson metric}.

We will introduce a slightly different way of labeling the splits, it will be useful for us later.
\begin{definition} \label{circ-system}
Let  $X$ be a set of nodes on a circle labeled clockwise by the symbols $1$ to  $n$. Define the dual  set $X^d$ consisting of nodes labeled by the symbols $\overline{1}$ to  $\overline{n}$ in such a way that each  $\overline{j}$ lies between $j$ and  $j+1.$ Then each chord  connecting $\overline{i}$ and $\overline{j}$ defines a split of $X$ which we will denote by $S_{ij}$.

It is obvious that the set $S_{ij}$ forms a circular split system which we will denote by $\mathcal{S}_c$. 

\end{definition}
The following theorem gives the formula for the weights of a weighted circular split systems \ref{BD}. 
\begin{theorem} \cite{GandC}, \cite{HKP} \label{th:decomp}
    Let $X$ be a set as in Definition \ref{circ-system} and let $D=(d_{ij})$ be a Kalmanson metric defined on $X$. Then the following split decomposition holds:
    \begin{equation} \label{decom}
        D=\sum \limits_{S_{ij} \in \mathcal{S}_c } \omega_{ij}D_{S_{ij}},
    \end{equation}
    where the coefficients  $\omega_{ij}$ are defined by the formula $$\omega_{ij}=\frac{1}{2}(d_{i,j}+d_{i+1,j+1}-d_{i,j+1}-d_{i+1,j}).$$
\end{theorem}
\begin{remark}\label{KW} The formula \eqref{decom} for $R_{ij}$ can be obtained from Theorem \ref{ken-wen} since taking the dual electrical network is an involution up to a shift: $\mathcal E^{**}=\mathcal E'$, where $\mathcal E'$ is obtained from $\mathcal E$ by shifting the numeration of the nodes of $\mathcal E$ by $1$ clockwise:

 \begin{align}  \label{formdual}
 R_{ij}=R^{**}_{i-1j-1}=-\sum \limits_{\bar{i}'<\bar{j}': \ D_{S_{i-1j-1}}(\bar{i}', \bar{j}') \neq 0}x^{*}_{\bar{i}'\bar{j}'}=-\sum \limits_{i'<j': \ D_{S_{ij}}(i', j') \neq 0}x^{*}_{i'j'}= \\ \nonumber
 \hspace{-40mm}    {\sum \limits_{i'<j': \ D_{S_{ij}}(i', j') \neq 0}\frac{1}{2}(R_{i,j}+R_{i+1,j+1}-R_{i,j+1}-R_{i+1,j})},
\end{align}  
where $\{1, \dots, n\}$ and $\{\overline1, \dots, \overline n\}$ are the labels of the nodes of $\mathcal E$ and $\mathcal E^{*}$ respectively. Thus  Theorem \ref{ken-wen} in fact provides an alternative way to observe that the effective resistance distance is the Kalmanson metric. 
\end{remark}
\begin{figure}[H]
\center
\includegraphics[width=80mm]{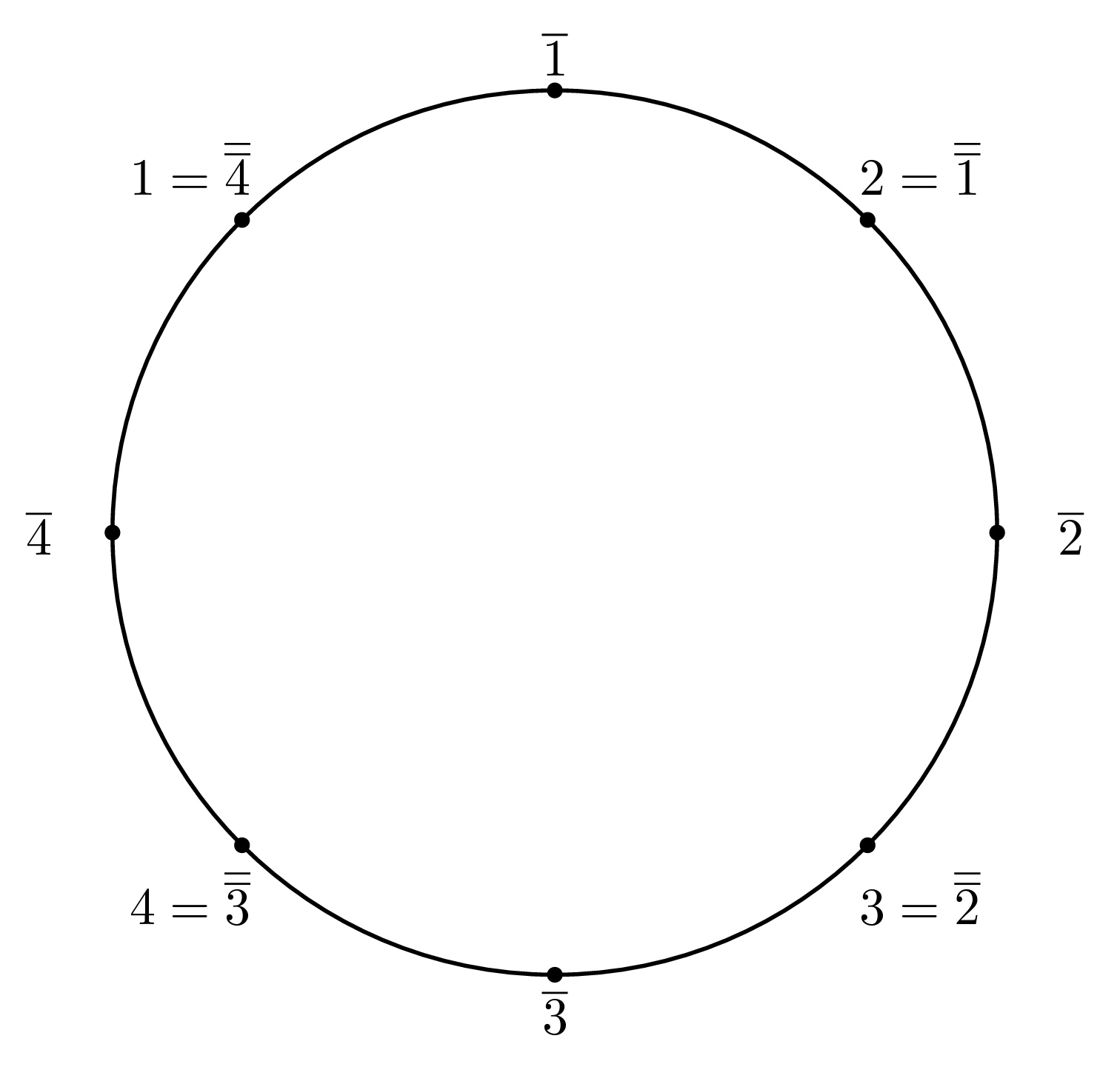}
\caption{Labeling in the formula \eqref{formdual} }
\label{fig:treedaul}
\end{figure}
\begin{definition}
Define a matrix $M(D)$
\begin{align}
    M(D)_{ij} = \begin{dcases*}
        \sum_{k \not= i}\omega_{ik}, & \text{if } i = j,\\
        -{\omega_{ij}}, & \text{if } $i \not= j,$  
        \end{dcases*}
\end{align}  
Notice that $\sum_{k \not= i}\omega_{ik}=-\omega_{ii}=d_{ii+1}.$
\end{definition}

Given a Kalmanson metric as a matrix \( D \), one can decide whether it is an electrical Kalmanson metric in a straightforward but computationally heavy way. First one must reorder the matrix columns and rows to ensure its entries satisfy the Kalmanson inequalities, then find the unique response matrix $M$ that corresponds to a given Kalmanson metric matrix using the pseudoinverse formula, see \cite{F} (Lemma 3.1), and then  check all the circular minors of that $M$ for non-negativity. 

Here is our main result, it gives a new characterization of the Kalmanson metrics, which are the resistance metrics of circular electrical networks without using inversion of matrices.
\begin{theorem} \label{th: dual}
  Let  $D$ be a Kalmanson metric matrix with the correct order of rows and columns, then $D$ is the effective resistance matrix of a connected circular electrical network $\mathcal E$ if and only if  the matrix $\Omega_{D}$ constructed from $D$ according to the  formula \eqref{eq:omega_n,r} defines a point $X$ in $\mathrm{Gr}_{\geq 0}(n-1, 2n)$ and the Plücker coordinate $\Delta_{24\dots 2n-2}(X)$ does not vanish.
\end{theorem}
\begin{proof}

Necessity follows from Theorem \ref{th:aboout omeganr}.

To prove sufficiency, assume that  $\Omega_{D}$  defines a point $X$ in $\mathrm{Gr}_{\geq 0}(n-1, 2n)$ and   $\Delta_{24\dots 2n-2}(\Omega'_{D}) \neq 0$. By direct computations we conclude that the matrix $\Omega_{D}s^{-1}$ has the form \eqref{omega_eq} and $\Delta_{13\dots 2n-3}\bigr((\Omega_{D}s^{-1})'\bigl)\neq 0$. Since the action of $s$ preserves the non-negativity according to \cite{L3}, the matrix $\Omega_{D}s^{-1}$ defines a point of  $\mathrm{Gr}_{\geq 0}(n-1, 2n)$. Using the surjectivity of Lam's embedding (see Theorem  \ref{about sur}) and Theorem \ref{th:aboout omeganr},  we obtain that both  $\Omega_{D}$ and  $\Omega_{D}s^{-1} $ are associated with connected networks.  

Finally, due to Theorem \ref{th:aboout omeganr} and Theorem \ref{ken-wen} we have that the matrix $M(D)=(m_{ij}),$ where   
\[m_{ij}=-\dfrac{1}{2}(d_{i,j}+d_{i+1,j+1}-d_{i,j+1}-d_{i+1,j})\]
can be identified with the response matrix of a network $\mathcal E^{*},$  associated with  $\Omega_{D}s^{-1}.$

 It remains  to prove that $d_{ij}$ are equal to the effective resistances $R_{ij}$ of the network $\mathcal E,$ associated with $\Omega_{D}.$ Since the metrics $d_{ij}$ and $R_{ij}$ are both Kalmanson and the weights in their split decomposition are equal, $\omega_{ij}=-m_{ij}$ (see the formula \eqref{decom}) , we conclude that:
 $$d_{ij}=-\sum \limits_{i'(i, j)j'(i, j)} m_{i'j'}= R_{ij}.$$
\end{proof}
The following theorem is also a convenient test for detecting electrical Kalmanson metrics. It literally falls in our lap due to Theorems \ref{ken-wen} and \ref{th:decomp}.


\begin{theorem} \label{th dual2}
Let $D$ be a Kalmanson metric on $X$,
then $D$ is an Electric Kalmanson metric of a connected circular electrical network $\mathcal E$ if and only if  the rank of $M(D)$ is equal to $n-1$ and the circular minors of the matrix $M(D)$ are non-negative after multiplying by $(-1)^k$ where $k$ is the size of the minor.

Given this, the matrix $M(D)$ is the response matrix of the dual to the electrical network $\mathcal E$. 
\end{theorem}

\begin{proof}

If any $k \times k$ circular minors of the matrix $M(D)$ is positive after multiplying by $(-1)^k$  
then $M(D)$ defines a response matrix of a network $\mathcal E'$ due to Theorem \ref{Set of response matrices all network}.
Hence we conclude that $\mathcal E'$ is connected (see Theorem \ref{Set of response matrices all network}) given the condition on the rank of $M(D)$.

The matrix $\Omega( \mathcal E')$ defines a point in $Gr_{\geq 0}(n-1, 2n)$ hence $\Omega(\mathcal E')s$ does as well. According to Theorem \ref{eq:omega_n,r} 
$$\Omega(\mathcal E')s=\Omega(\mathcal E)=\Omega_{R}(\mathcal E)$$ 
for a connected network $\mathcal E$ with the resistance matrix $R_{\mathcal E}=M(D)$, moreover $\mathcal E'=\mathcal E^{*} $. It remains to prove that $R_{ij}=d_{ij}, $ it can be done as it has been explained in the proof of Theorem \ref{th: dual}.

Suppose now that $D$ comes from a circular electrical split system associated with a connected circular electrical network $\mathcal E$, then due to Theorem \ref{eq:omega_n,r}  we conclude that $\Omega_{R}(\mathcal E)s^{-1}$ defines a point in $Gr_{\geq 0}(n-1, 2n)$ associated with a connected network $\mathcal E^{*}$ and $M(D)=M_R(\mathcal E^*).$ The properties of $M(D)$ in the statement of the theorem follow from Theorem \ref{Set of response matrices all network}.
\end{proof}
A few remarks are in order.
\begin{remark}
Many statements in this paper can be extended to the connected cactus networks. In particular, one can obtain that the effective resistances  of a connected cactus network give rise to a Kalmanson pseudometric and a Kalmanson pseudometric matrix $D$ can be identified with the resistance matrix of a connected cactus networks if and only if   the matrix $\Omega_{D}$ constructed from $D$ according to the  formula \eqref{eq:omega_n,r} defines a point $X$ in $\mathrm{Gr}_{\geq 0}(n-1, 2n)$.
\end{remark}
\begin{remark}\label{nk}
In Remark 4.3, [F] S. Forcey gives an example of a non-planar network whose resistance metric is nevertheless Kalmanson. It is easy to generalise this example in the following way. Let's start with a circular planar electrical network, whose resistance metric therefore is Kalmanson, such that all the inequalities required by the Kalmanson condition are strict. This means the resistance metric is represented by a full circular planar split system. Add an edge with a variable conductance $x$ to it in such a way that it is not a circular planar anymore, as in the Figure \ref{nkal}.
It is clear from the Kirchhoff formula \ref{KF} that each effective resistance $R_{ij}$ changes as a function of $x$ according to the formula
$$R^{\text new}_{ij}=\frac{R_{ij}+r_{ij}x}{1+Tx}$$
where $r_{ij}$ and $T$ are non-negative numbers and $T$ does not depend on $i,\,j$.  Therefore, for small $x$ the Kalmanson condition for the resistance metric will hold for the new network, although it is not any longer circular planar. This means that the planar circular Kalmanson metrics are sitting far away from the boundary of the space of all Kalmanson metrics.
\end{remark}
\begin{figure}
\center
\includegraphics[width=50mm]{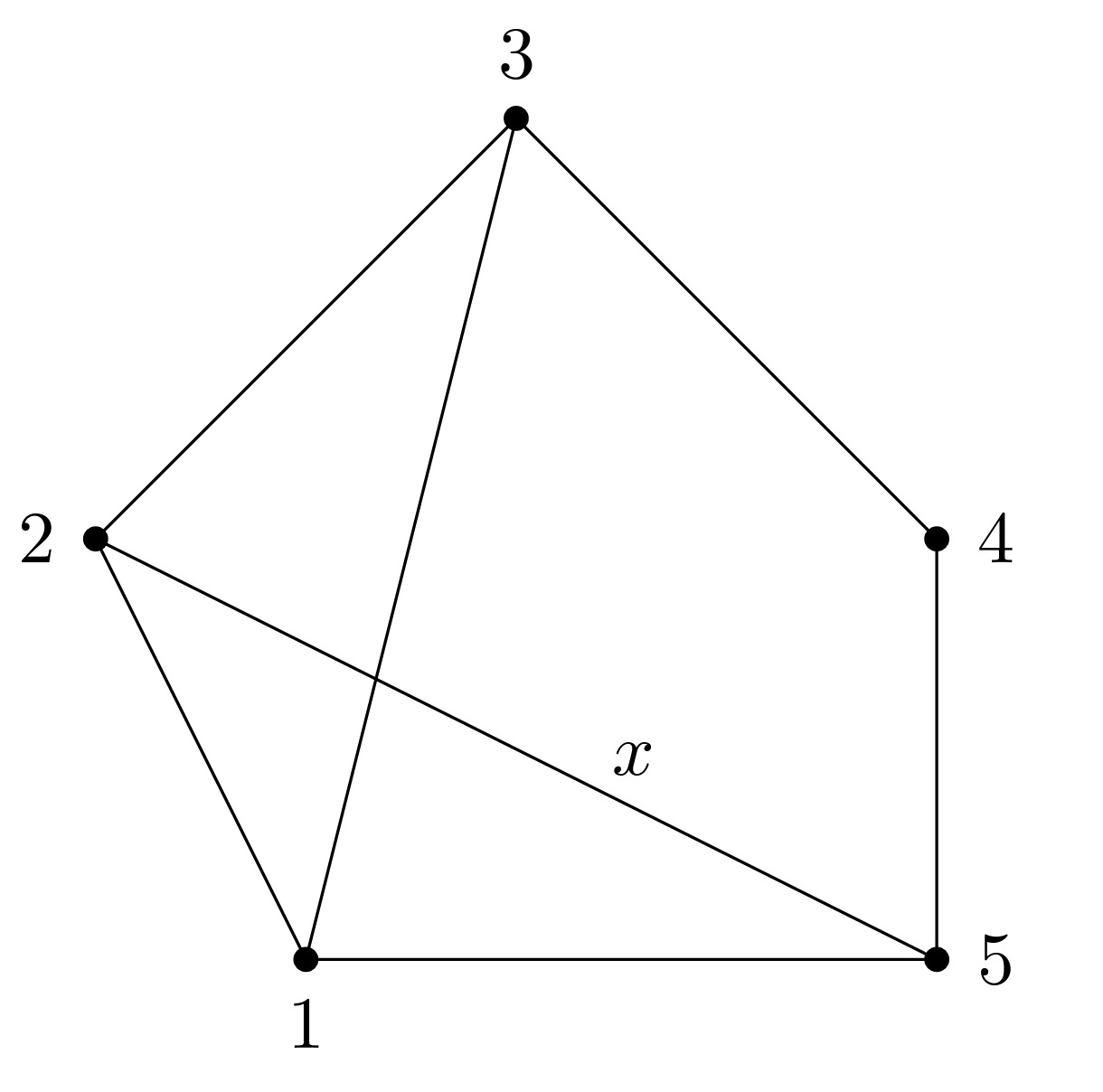}
\caption{A non-planar electrical network with nodes \{1,\dots,5\} }
\label{nkal}
\end{figure}
\begin{remark}\label{plueclerformetric}
 Let $X$ and $D$ be as in Theorem \ref{th:decomp}, then the matrix $M(D)$ defined above is the response matrix of a connected electrical network, not necessarily circular planar. Namely, it is symmetric, all non-diagonal elements are negative and the sum of the elements in any row is zero.
 The connection of the resistance matrix of this network to the original metric $D$ is an interesting question.
 
 Moreover, such a matrix defines a point 
 $$\mathrm{IG}(n-1, 2n)\subset \mathrm{Gr}(n-1, 2n)$$
 according to Theorem \ref{th:aboout omeganr}. Its Plücker coordinates are interesting invariants of the metric $D$. 
We are planning to address these questions in a future publication.
\end{remark}
\begin{remark} The resistance metric has the following important property: its
square root $\sqrt {D(i,j)}$ is $L_2$ embeddable \cite{K}, hence the electric Kalmanson metrics have that property. There is a well known condition for $L_2$ embeddability stated in terms of the minors of the Caley-Menger matrix \cite{GandC}. It is interesting to see if this condition also describes the electric Kalmanson metrics in the set of all Kalmanson metrics.
\end{remark}
\begin{figure}[h!]
\center
\includegraphics[width=70mm]{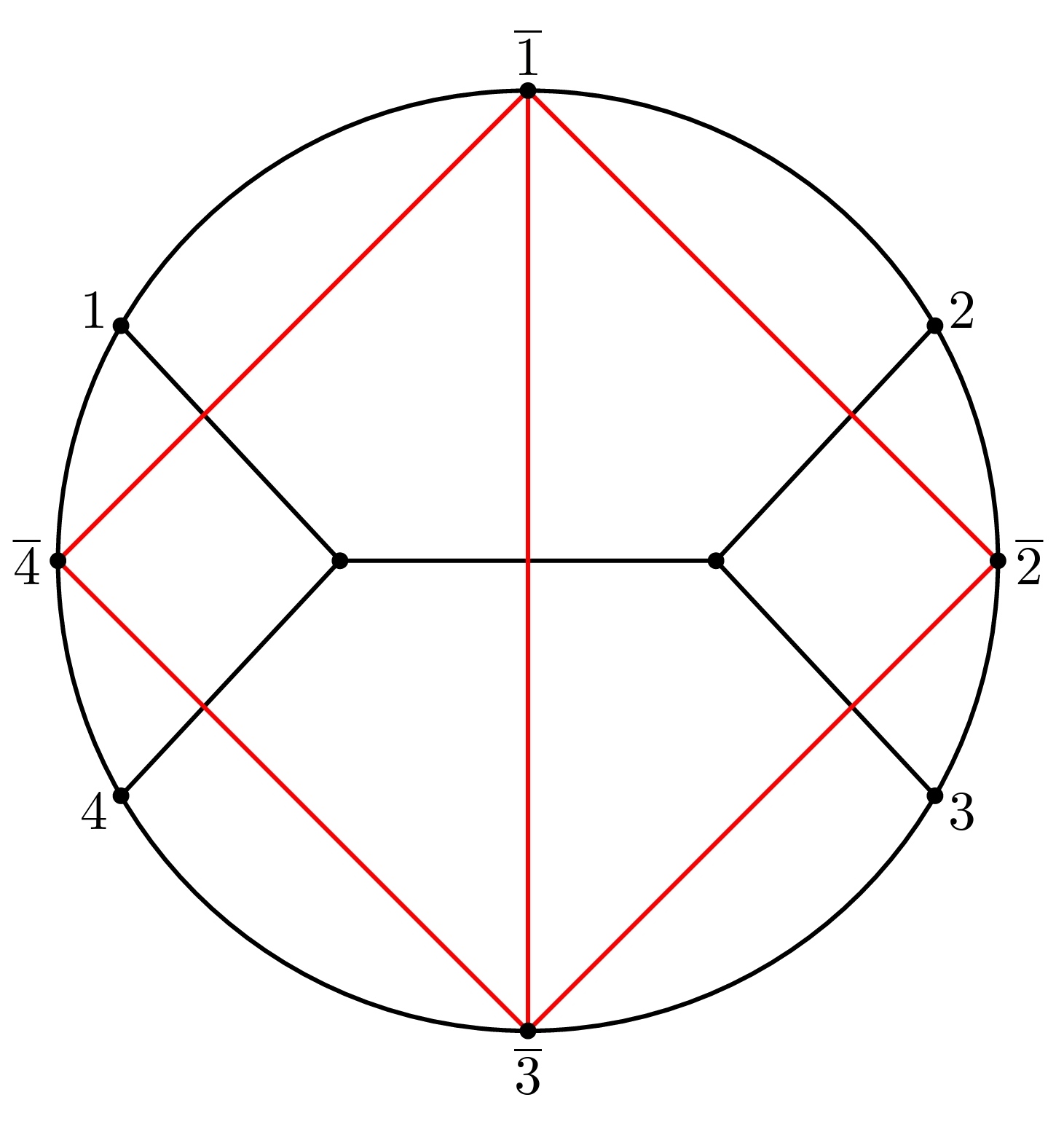}
\caption{A tree network $\mathcal E$ and its dual  network $\mathcal E^{*},$ \\ all the conductances are equal to $1$}
\label{treedaul}
\end{figure}
We will provide an example to illustrate the theorem above.
\begin{example} 
Let $T$ be a tree with four leaves as in the Figure \ref{treedaul} with the weights of edges all equal to one. The resistance metric matrix of the tree is 
\[D=
\begin{pmatrix}
0& 3 & 3& 2 \\
3&  0& 2& 3\\
3& 2&  0& 3 \\
2&3&3&0
\end{pmatrix}
\]
It coincides with the matrix of the geodesic distance for trees. Its split decomposition is as follows as one sees easily
\[D=D_{S_{12}}+D_{S_{13}}+D_{S_{14}}+D_{S_{23}}+D_{S_{34}}+D_{S_{43}}=D_{2|134}+D_{23|14}+D_{1|234}+D_{3|124}+D_{4|123}\]
The response matrix of the dual network multiplied by $-1$ is
\[
\begin{pmatrix}
-3& 1& 1& 1 \\
1&  -2& 1& 0\\
1& 1& -3 & 1 \\
1&0&1&-2
\end{pmatrix}
\]
It coincides with the incidence matrix of the dual graph since there are no internal vertices.
Our correspondence between the splits and pairs of numbers is given in the table
\\
\[\begin{tabular}{ |c|c|c|c|c |}
 $(\bar1\bar2)$& $(\bar1\bar3)$& $(\bar1\bar4)$& $(\bar2\bar3)$&$ (\bar3\bar4)$ \\
 \hline
 $(2|134)$ & $(14|23)$&$ (1|234)$&$(3|124)$&$(4|123)$ \\
\end{tabular}\]
\\
matching the coefficients in the split decomposition of $D$ and the coefficients of the response matrix of the dual network.

\end{example}


\section{Reconstruction of network topology} \label{sec:rec}
As we mentioned in the introduction one of the important problems in applied mathematics can be formulated as follows:
\begin{problem} \label{black-box}
  Suppose we are given a matrix $D=(d_{ij})$, whose entries are the distances between $n$ terminal nodes of an unknown weighted graph $G$. It is required to recover the graph $G$  and the edge weights which are consistent with the given  matrix.       
\end{problem}
If the terminal nodes  are the boundary vertices of a circular electrical networks $\mathcal E $ and $D=R_{\mathcal E}$, then due to Proposition \ref{th: about inverse resp} we conclude that   Problem \ref{black-box} can be identified with the {\it black box problem}, see \cite{CIW}.
It is also known as the discrete Calderon problem or the discrete inverse electro impedance tomography problem.

If  a graph $G$ is an unknown tree $T$ and the entries of $D_T=(d_{ij})$ are equal to the weights of the paths between vertices of $T$, then Problem \ref{black-box} can be identified with the minimal tree reconstruction problem, which plays an important role in phylogenetics \cite{HRS}.
\begin{definition}
    We will call a  matrix $D=(d_{ij})$ a {\it tree realizable}, if there is a tree $T$ such that 
    \begin{itemize}
        \item $D=D_T$   i.e. $d_{ij}$ are  equal to the weights of the paths between terminal nodes of $T$;
        \item  a set of terminal nodes contains all leafs of $T;$
        \item  there are no vertices of degree $2.$ 
    \end{itemize}
    
\end{definition}

\begin{theorem} \cite{CR}, \cite{HY} \label{minimaltree}
 If a  matrix $D$ is tree realizable, then there is an unique  tree $T$ such that $ D=D_T.$

\end{theorem}

It is not difficult to see that if the graph of a circular electrical network $\mathcal E $ is a tree $T$ and its  boundary nodes contain all tree leafs, then $R_{\mathcal E}=D_{\bar T}$ where $\bar T$ and $T$ are identical as unweighted trees and the weights of the corresponded edges are reciprocal. This observation allows us to use in phylogenetics the algorithm for reconstruction of electrical networks \cite{L} from a given resistance matrix.  Our reconstruction method is different from the methods suggested in \cite{F}, \cite{FS}. 

\begin{definition}
    The median graph of a circular network $\mathcal E $ with the graph $\Gamma$ is the graph $\Gamma_M$ whose internal vertices are the midpoints of the edges of $\Gamma$ and two internal vertices are connected by an edge if the edges of the original graph $\Gamma$ are adjacent in one face. The boundary vertices of $\Gamma_M$ are defined as the intersection of the natural extensions of the edges of $\Gamma_M$  with the boundary circle.  Since the interior vertices of the median graph have degree four, we can define the strands of the median graph as  the  paths which always go straight through any degree four vertex. 

The strands naturally define a permutation $\tau(\mathcal E)$ on the set of points $\{1, \dots, 2n\}.$
\end{definition}

\begin{figure}[h!]
    \centering
    \includegraphics[width=0.4\textwidth]{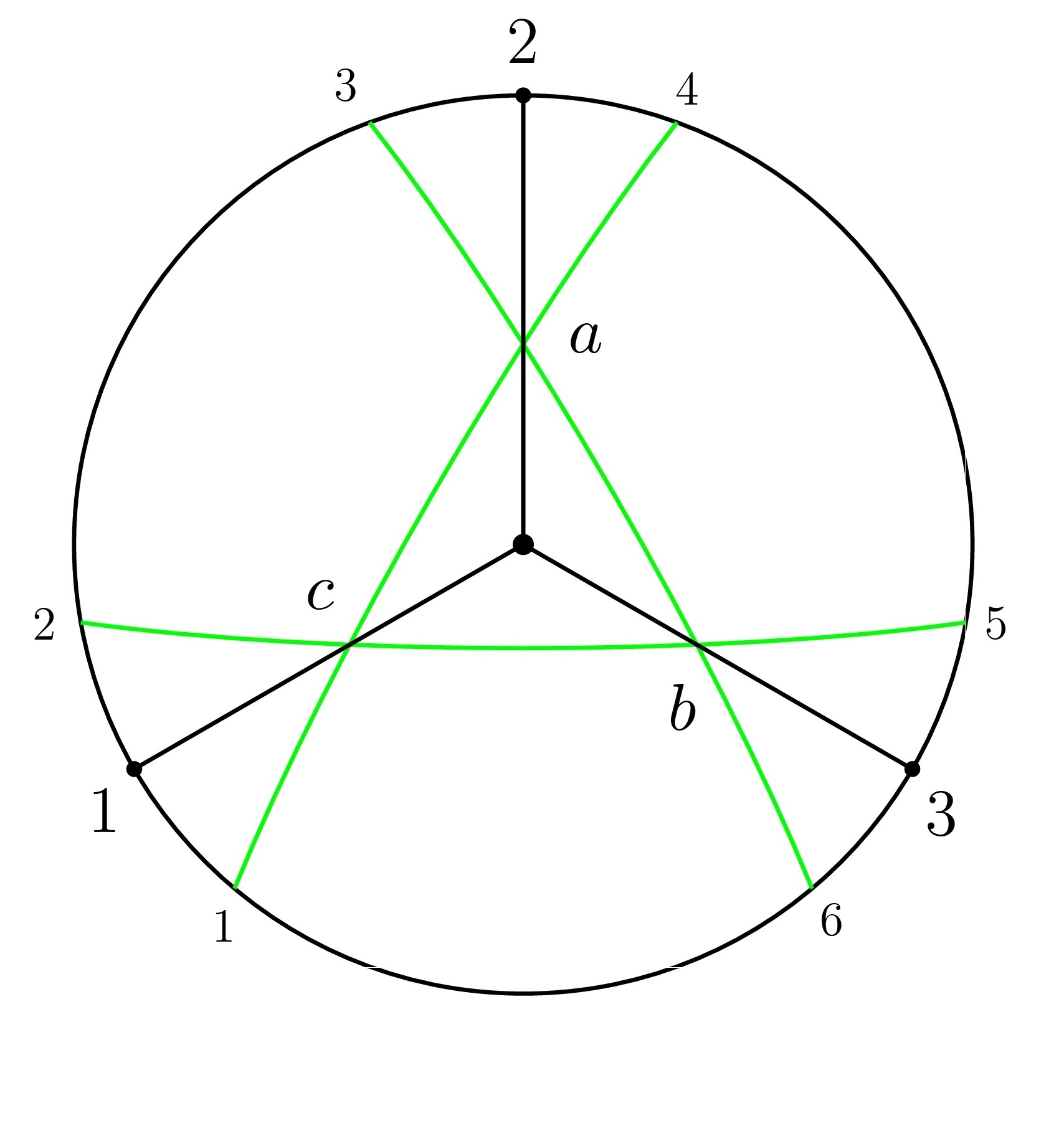}
    \caption{Star-shape network, its median graph and the strand permutation $\tau(\mathcal E)=(14)(36)(25)$ }
    \label{fig:triangle}
\end{figure}
\begin{theorem} \cite{CIW} \label{strand}
     A circular electrical network is defined uniquely up to electrical transformations by its strand permutation.
\end{theorem}
 Denote by $A_i$ the columns of the matrix $\Omega_{ R}(\mathcal E)$ and define the column permutation $g(\mathcal E)$ as follows:
    $g(\mathcal E)(i)=j,$ if  $j$  is the minimal number such that   $A_i \in \ \mathrm{span}(A_{i+1}, \dots, A_{j}  ),$ where the indices are taken modulo $2n$.
\begin{theorem}\label{permutationj} \label{th:aboutstradsper}
    The following holds $$g(\mathcal E)+1=\tau(\mathcal E).$$ 
\end{theorem}
     \begin{proof}
         This result follows from the fact that $\Omega_{ R}(\mathcal E)$ is an explicit parametrization of the Lam embedding, see Proposition $5.17$ \cite{L}. 
     \end{proof}
\begin{definition}
A circular electrical network is called minimal if the strands of its median graph do not have self-intersections; any two strands intersect at most one point and the median graph has no loops or lenses, see the Figure \ref{fig:loop}.
\end{definition}


\begin{figure}[h!]
    \centering
    \includegraphics[width=0.3\textwidth]{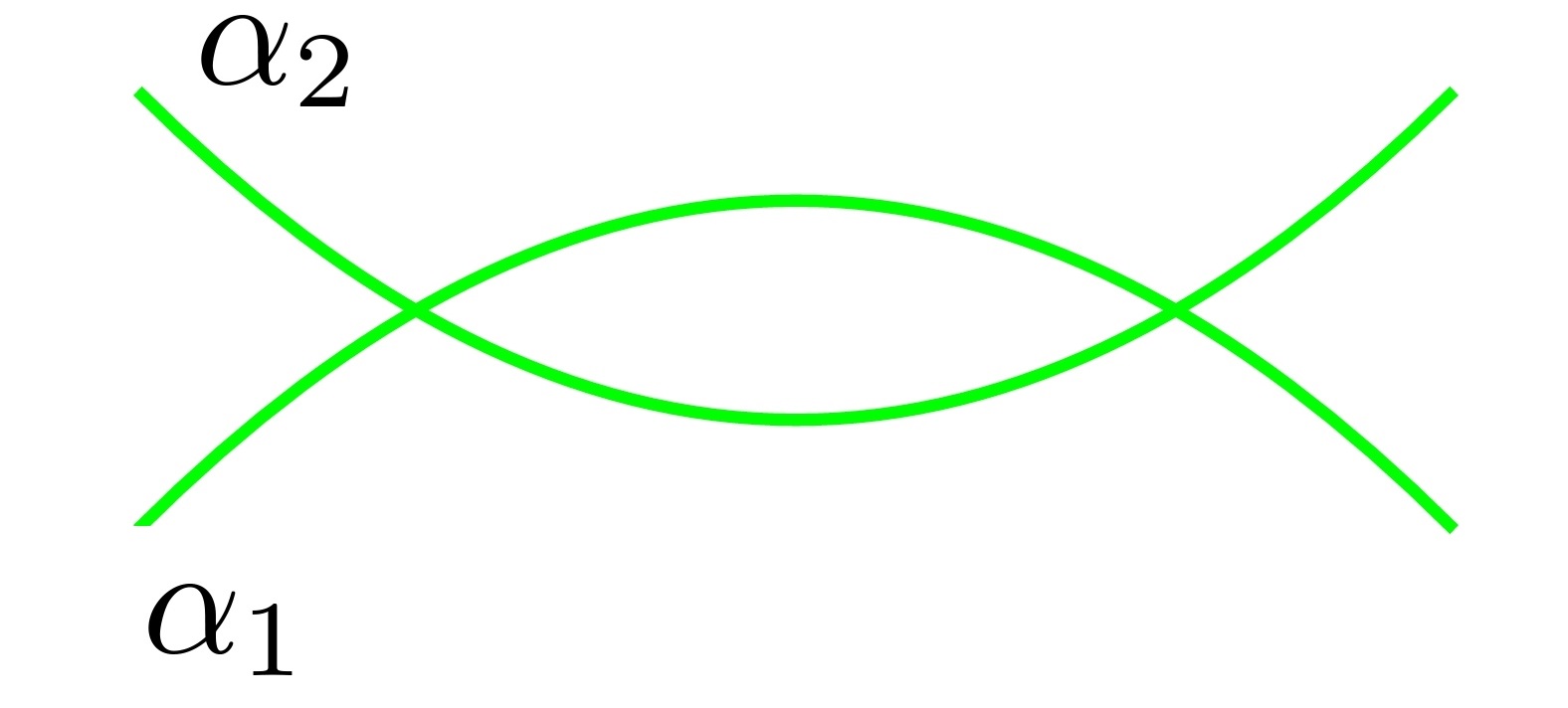}
    \caption{A lens obtained  by the intersection of strands  $
\alpha_1$ and  $
\alpha_2$ }
    \label{fig:loop}
\end{figure}

Based on Theorem \ref{permutationj} we suggest the following reconstruction algorithm (see Example \ref{ex:rec-ex}):
\begin{itemize}
    \item For a given matrix $R_{\mathcal E}$ construct the matrix $\Omega_{ R}(\mathcal E);$
    \item Using $\Omega_{ R}(\mathcal E)$ calculate a strand permutation $\tau(\mathcal E)$;
    \item The permutation $\tau(\mathcal E)$ defines a strand diagram, which can be transformed to a median graph of a \textit{minimal} circular electrical network $\mathcal E$ using the procedure described in   \cite{CIW}. 
    \item From the median graph we recover the  network $\mathcal E$ as in \cite{CIW} or \cite{F}.
\end{itemize}

\begin{theorem} \cite{CIW}\label{electrominim}
Any circular electrical network is equivalent to a minimal network. 

Any two minimal circular electrical networks which share  
the same response matrix, and hence the same effective resistance matrix, can be converted to each other only by the star-triangle transformation.  As a consequence we obtain that  any two equivalent minimal circular electrical networks have the same number of edges which is less or equal than $\frac{n(n-1)}{2}$.

If two minimal circular electrical networks are equivalent,  they have 
the same strand permutation.
\end{theorem}
Informally, the theorem says that removing all the loops, the internal vertices of degree $1$ and reduction of all the parallels and the series  any network can be converted into a minimal network.
   
Let us compare the last result with Theorem  \ref{minimaltree}.
\begin{theorem} \label{tomin}
Let $\mathcal E$ be a minimal electrical network whose graph is a tree $T$, then such a tree is unique.
   
  Let $\mathcal{E}(T, \omega)$ be a circular planar electrical network on a tree $T$ such that all its leaves are nodes and there are no inner vertices of degree $2$, then $\mathcal{E}(T, \omega)$ is minimal.
\end{theorem}
\begin{proof}
Indeed, it is easy to see that $R_T$ is a tree realizable by $\bar T$. Therefore, $T$ is unique by Theorem  \ref{minimaltree}. This proves the first statement.

To prove the second statement, it is enough to verify that the median graph $T_M$ of $\mathcal{E}(T, \omega)$ has no lenses. Indeed, each strand of $T_M$ corresponds to a simple path between nodes in the original graph, see the Figure \ref{fig:treepath}. Since $T$ has no inner vertices of degree $2$ these paths can intersect at most once. Therefore, $T_M$ has no lenses and $\mathcal{E}(T, \omega)$ is minimal.
\end{proof}
  \begin{figure}[h!]
\center
\includegraphics[width=100mm]{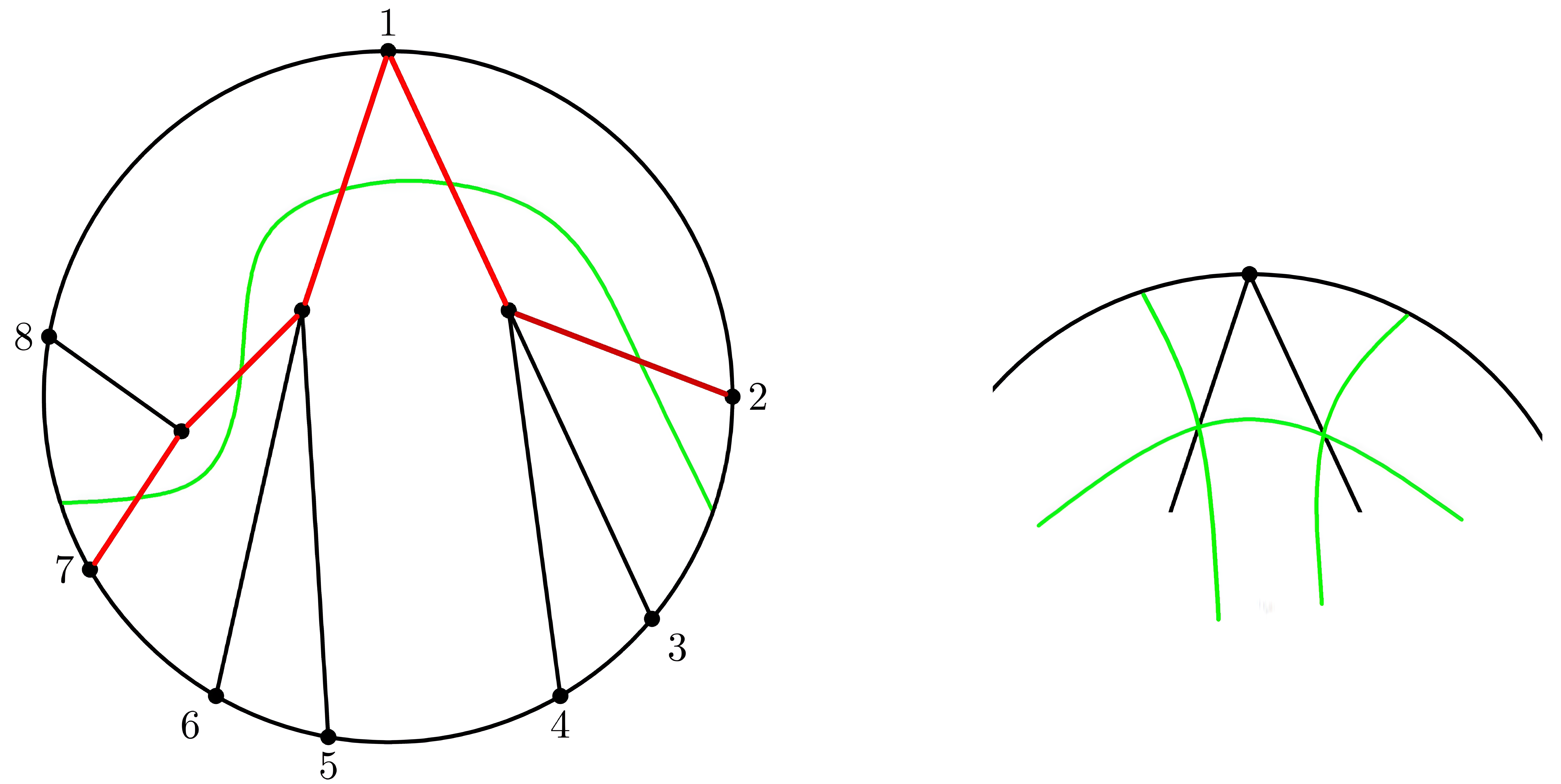}
\caption{On the left: A path which corresponds to a strand; On the right: the stands around a node of the degree $2$}
\label{fig:treepath}
\end{figure}
We propose  the following  algorithm for reconstruction of the minimal tree for a given tree metric $D_T$ based on Theorem \ref{tomin}:
\begin{itemize}
    \item Do all steps of the algorithm described above to obtain  a minimal network $\mathcal E$ such that $R_{\mathcal E}=D_T$;
    \item Transform $\mathcal E$ to a minimal tree by a sequence of the star-triangle transformation. 
\end{itemize}
Heuristically, by (\cite{Bu}, Theorem $1$), we assume that the second step can be performed monotonically, converting all triangles  into stars.
\begin{remark}
    A more advanced technique called the chamber ansatz when applied to $\Omega_{ R}(\mathcal E)$ gives an algorithm for recovering not only the topology of the network but the weights of the edges as well. This method is described  in \cite{Ka}.  
\end{remark}

We will illustrate our algorithm by an example.

\begin{figure}[H]
    \centering
    \includegraphics[width=1.2\textwidth]{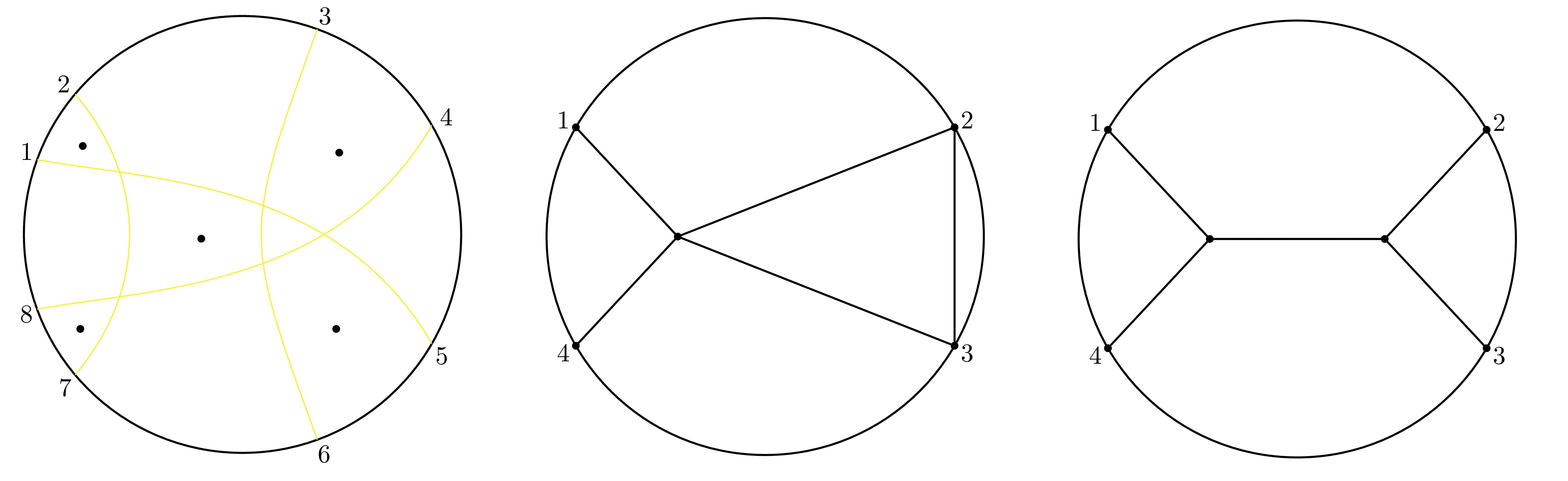}
    \caption{ An example of a reconstruction of a network topology}
    \label{fig:recon}
\end{figure}
\begin{example} \label{ex:rec-ex}
    Consider a dissimilarity matrix
    \[D=
\begin{pmatrix}
0& 3 & 3& 2 \\
3&  0& 2& 3\\
3& 2&  0& 3 \\
2&3&3&0
\end{pmatrix}
\]
Then the matrix $\Omega'_{ D}$ has the form
\[\Omega'_D=
\begin{pmatrix}
1& 3 & 1& 1 & 0 & -1& 0 &1 \\
0& 1 & 1& 2 & 1 & 1& 0 &0\\
0& -1 & 0& 1 & 1 & 3& 1 &1 
\end{pmatrix}
\]
\end{example}
and satisfies the Theorem \ref{th: dual}, therefore  $\Omega'_{ D}=\Omega'_R(\mathcal E)$ for a network $\mathcal E \in E_4.$

By direct computations we verify that $g(\mathcal E)=
[4\,6\,5\,7\,8\,2\,1\,3]$ in the one window notation as it is used in \cite{L}, therefore the strand permutation $\tau(\mathcal E)$ is $\tau(\mathcal E)=(15)(27)(36)(48).$ This  strand permutation defines a minimal network as it is shown in the Figure \ref{fig:recon}, which can be transformed to a tree by applying one star-triangle transformation. 
\


\begin{thebibliography}{9999999}
\bibitem{BD} H-J. Bandelt and A. Dress, Split decomposition: a new and useful approach to phylogenetic analysis of distance data, Molecular Phylogenetics and Evolution 1 (1992) 242–252.
\bibitem{BD1} H-J. Bandelt and A. Dress, A canonical decomposition theory for metrics on a finite set, Advances in Mathematics
Volume 92, Issue 1, March (1992) Pages 47-105.
\bibitem{Bu} P. Buneman,  A note on the metric properties of trees. J. Combin. Theory Ser. B.,  Vol. 17,  No. 1, pp. 48-50 (1974).
\bibitem{BGGK} B. Bychkov, V. Gorbounov, L. Guterman, A. Kazakov,  Symplectic geometry of electrical networks, Journal of Geometry and Physics
Volume 207, January 2025, 105323.
\bibitem{BGKT} B. Bychkov, V. Gorbounov, A. Kazakov, D. Talalaev, Electrical Networks, Lagrangian Grassmannians, and Symplectic Groups, Moscow Mathematical Journal, Vol.23, No.2, pp.133-167 (2023).
\bibitem{BGK} B. Bychkov, L. Guterman, A. Kazakov,  Electrical networks via circular minors, in preparation.
\bibitem{CGS} S. Chepuri, T. George and D. E. Speyer, Electrical networks and Lagrangian Grassmannians, https://arxiv.org/abs/2106.15418 (2021).
\bibitem{CR} J. C. Culberson, P. Rudnicki, A fast algorithm for constructing trees from distance matrices. Information Processing Letters, Vol. 30, No. 4, pp. 215-220 (1989).
\bibitem{CIM} E. B. Curtis, D. Ingerman, J. A. Morrow, Circular planar graphs and resistor networks. Linear algebra and its applications, Vol.283, No. 1-3, pp.115-150 (1998).
\bibitem{CIW} B. Curtis, J. A. Morrow, Inverse problems for electrical networks. World Scientific. Vol.13 (2000).

\bibitem{GandC} M. Deza, M. Laurent, Geometry of Cuts and Metrics, Springer Berlin, DOI https://doi.org/10.1007/978-3-642-04295-9.
\bibitem{CdV} Colin de Verdi`ere, Yves  réseaux électriques planaires. I. Comment. Math. Helv. 69 (1994), no. 3, 351374.
\bibitem{DF} S. Devadoss, S. Forcey, Compactifications of phylogenetic systems and electrical networks, https://doi.org/10.48550/arXiv.2408.03431.
\bibitem{DP} S. L. Devadoss, S. Petti,  A space of phylogenetic networks. SIAM Journal on Applied Algebra and Geometry, Vol. 1, No. 1,  pp. 683-705 (2017).
\bibitem{F} S. Forcey, Circular planar electrical  networks, split systems, and phylogenetic networks, https://doi.org/10.48550/arXiv.2108.00550.
\bibitem{FS} Forcey S, Scalzo D. Phylogenetic Networks as Circuits With Resistance Distance. Front Genet. 2020 Oct 15;11:586664. doi: 10.3389/fgene.2020.586664. PMID: 33193721; PMCID: PMC7593533.
\bibitem{HY} S. L. Hakimi, S. S. Yau,  Distance matrix of a graph and its realizability. Quarterly of applied mathematics, Vol. 22, No. 4, pp. 305-317 (1965).
\bibitem{HRS} D. H. Huson, R. Rupp,  C. Scornavacca, Phylogenetic networks: concepts, algorithms and applications. Cambridge University Press (2010).
\bibitem{Ka} A. A. Kazakov,  Inverse problems related to electrical networks and the geometry of non-negative Grassmannians. arXiv preprint arXiv:2502.16710. (2025).

\bibitem{MTT} G. Kirchhoff, On the solution of the equations obtained from the investigation of the linear distribution of galvanic currents. IRE Transactions on Circuit Theory, 5(1), 4-7 (1958).
\bibitem{KW 2011} R. W. Kenyon and D. B. Wilson, Boundary partitions in trees and dimers, Trans. Amer. Math. Soc., 363, pp. 1325–1364  
\bibitem{Kl} D.J. Klein, M. Randić,  Resistance distance. J Math Chem 12, 81–95, https://doi.org/10.1007/BF01164627 (1993).
\bibitem{K} D. Klein, H. Zhu,  Distances and volumina for graphs. Journal of Mathematical Chemistry 23, 179–195 (1998). https://doi.org/10.1023/A:1019108905697
\bibitem{HKP} A. Kleinman, M. Harel, L. Pachter, Affine and Projective Tree Metric Theorems. Ann. Comb. 17, 205–228 (2013). https://doi.org/10.1007/s00026-012-0173-2.
\bibitem{L} T. Lam, Electroid varieties and a compactification of the space of electrical networks, Adv. Math. 338, pp. 549–600 (2018).
\bibitem{L3} T. Lam,  Totally nonnegative Grassmannian and Grassmann polytopes, arXiv preprint arXiv:1506.00603 (2015).
\bibitem{MS} M. A. Steel, Phylogeny: Discrete and Random Processes in Evolution, 2016, https://api.semanticscholar.org/CorpusID:31002469.
\bibitem{SS} D. Speyer, B. Sturmfels, The tropical Grassmannian. Adv. Geom. 4 (2004), 389-411.
\bibitem{Wa} D. G. Wagner, Combinatorics of electrical networks, Lecture Notes, Dept. of C and O, University of Waterloo (2009).
\bibitem{Z} A. Zelevinsky, "What Is . . . a Cluster Algebra?", AMS Notices, 54 (11): 1494–1495, 2007.

\end{thebibliography}
\end{document}